\documentclass[12pt]{article}

\usepackage{amsmath, bm}
\usepackage{amssymb}
\usepackage{amsthm}

\usepackage{mathtools}
\usepackage{enumerate}

\usepackage{enumitem} 
\usepackage{graphicx}

\usepackage{placeins} %floatbarrier

\usepackage{geometry}
\newcommand{\R}{\mathbb{R}}
\newcommand{\T}{\mathbb{T}}
\newcommand{\E}{\mathbb{E}}
\newcommand{\C}{\mathbb{C}}
\newcommand{\N}{\mathbb{N}}
\newcommand{\Z}{\mathbb{Z}}

\newcommand{\mch}{\mathcal{H}}
\newcommand{\mcf}{\mathcal{F}}
\newcommand{\mcs}{\mathcal{S}}

\newcommand{\bldA}{\mathbf{A}}
\newcommand{\bldv}{\mathbf{v}} 
\newcommand{\bldc}{\mathbf{c}}
\newcommand{\bldf}{\mathbf{f}}
\newcommand{\bldg}{\mathbf{g}}

\newcommand{\bldeta}{\bm{\eta}}

\newcommand{\supp}{\text{supp}\,}

\DeclareMathOperator{\sep}{sep}
\DeclareMathOperator{\td}{d\!}

\DeclareMathOperator{\dist}{dist}
\DeclareMathOperator{\PW}{PW}
\DeclareMathOperator{\Arg}{Arg}

\newtheorem{theorem}{Theorem}[section]
\newtheorem{lemma}[theorem]{Lemma}

\newtheorem{corollary}[theorem]{Corollary}
\newtheorem{example}[theorem]{Example}

\theoremstyle{remark}
\newtheorem*{remark}{Remark}

\title{Well-Posedness of Sparse Frequency Estimation}
\author{
	Benedikt Diederichs\thanks{University of Passau and Fraunhofer IIS Research Group “Knowledge Based Image Processing”, Passau, Germany} 
}
\date{\today}

\providecommand{\keywords}[1]{\textbf{Key words.} #1}
\providecommand{\amssubjects}[1]{\textbf{AMS subject classifications.} #1}

\begin{document}

\maketitle

\begin{abstract}
	The problem of estimating the frequencies of an exponential sum has been studied extensively over the last years. It can be understood as a sparse estimation problem, as it strives to identify the sparse representation of a signal using exponentials. In this paper, we are interested in its intrinsic stability properties. We derive a bound very similar to the restricted isometry property. Conditional well-posedness follows: Any exponential sum with samples close to the unknown ground truth has close frequencies as well, provided that it satisfies our model assumptions. The most important assumption is that the frequencies are well-separated. Furthermore, we show that the presented bound is sharp and gives rise to improved estimates of condition numbers of certain Vandermonde matrices.
\end{abstract}
\amssubjects{65T40, 42C15, 15A18}

\section{Introduction}

Spectral analysis of a given signal is one of the most ubiquitous problems in signal processing.
Depending on the problem at hand, one can choose from a large collection of well-studied methods. 

In this paper, we are interested in one of the most important special cases, where the signal satisfies the harmonic model, i.e., 
$$f(x) = \sum_{y\in Y^f} c_y e^{2\pi i y x}, \quad c_y\in\C\setminus\{0\}. $$
Here, $Y^f\equiv Y\subset \T = \R/\Z \simeq [0,1)$ is the finite set of frequencies of $f$, with coefficients (or amplitudes) $c_y\in \C\setminus \{0\}$. Note that we enumerate $c^f\equiv c = (c_y)_{y\in Y}\in \C^Y$ by the set $Y$, which is convenient if one does not want to fix an enumeration of $Y$.

Such an $f$ is then sampled at integers $k=0,\dots, N$, and we wish to estimate $Y$ and the coefficients $c$. Typically, the available samples are corrupted by noise. 

This problem has been studied for a long time, and many relevant applications are known, see for example the textbook  \cite{manolakis2000statistical}. It reemerged over the last years for a number of reasons. The first is that using the harmonic model allows to overcome resolution limits of generic methods, i.e., super-resolution is possible. And while that is well known, new ideas are further expanding the possibilities in applications.

More interesting for this paper is the connection to compressed sensing. As $f$ is sparse in the frequency domain (its spectrum is a sum of Dirac deltas), the frequency estimation problem can be seen as finding a sparse representation of $f$ from an uncountable, highly correlated dictionary consisting of all functions $x\mapsto e^{2\pi i xy}$. Often, the problem is discretized, see for example \cite{hu2012compressed, duarte2013spectral}. An alternative is an approach using TV-minimization, as promoted in \cite{candes2013super, candes2014towards, duval2015exact} and others. 

In contrast to the existing literature, we are interested in the intrinsic stability properties of the continuous problem. We prove that for any two exponential sums satisfying our model assumption, close samples imply close frequencies. The model assumptions include most notably a separation condition of the frequencies. That bound can be interpreted as a continuous analog of the restricted isometry property in compressed sensing and is valid regardless of a solution method.

As the connection to compressed sensing is of central importance for the paper, we give an informal explanation and motivation in the next section. After that, we state and prove the main results in Section \ref{secLocFuncAndStab}. In the last two sections we present some corollaries: In Section \ref{secSingValVander} we give lower bounds for singular values of Vandermonde matrices and in Section \ref{secAPostEst} we finally state the well-posedness of the problem and a posteriori error bounds.

To keep notation simple, we use $a\lesssim b$, which means that $a\leq  C b$, where $C$ is a positive constant, independent of any quantities of interest (on which $a,b$ might depend). Further, $a\sim b$ means $a\lesssim b$ and $b\lesssim a$. On the other hand, $a \approx b$ is used as an informal notation for $a$ and $b$ being close.

We call an vector $s$-sparse, if all but $s$ of its entries are zero. The Fourier transform is defined by
\[\mcf f (w) = \hat{f}(w) = \int_\R f(x)e^{-2\pi i x w} \td x.\]

\section{Connection to Sparse Estimation}
We start by formulating an abstract sparse estimation problem. We are interested in  identifying 
\begin{align*}
	x = \sum_{\theta \in Y} c_{\theta} h(\theta)\in\mch.
\end{align*} 
Here, $Y\subset \Theta$ is a finite subset of $\Theta\subset\R^d$, the set of valid parameters. $c_{\theta}\in\C\setminus\{0\}$ are the coefficients, $\mch$ is a Banach space and $h:\Theta \rightarrow \mch$ is a continuous and injective function. 

Given are a finite number of measurements of the form
\begin{align*}
	b_j = a_j(x)+\eta_j,\quad j=1,\dots, m,
\end{align*}
where $a_j\in\mch'$ are elements of the dual space and $\eta_j$ is noise, corrupting the measurements.

The frequency estimation problem fits into that framework by choosing 
$$h:\T \rightarrow C(I),\quad h(y) = e^{2\pi i y \cdot},$$
where $I\subset \R$ is an interval containing all sampling points. $a_j$ are point evaluations at the $j$th sampling point and $C(I)$ is the Banach space of continuous functions from $I$ to $\C$.

To solve this problem, one typically discretizes $\Theta$ by replacing it with $\tilde{\Theta}=\{\tilde{\theta}_1,\dots, \tilde{\theta}_N \}$. Letting $a_{jk} = a_j(\tilde{\theta}_k)$ and $A=(a_{jk})$ and assuming $\theta_j\in\tilde{\Theta}$, we arrive at the typical problem of compressed sensing, namely
$$Ax = b, $$
where $b=(b_1,\dots, b_m)\in\C^m$. We identified $x$ with the vector $x\in\C^N$, $x_k = c_j$ if $\tilde{\theta}_k = \theta_j$ and zero otherwise. Now as $N\gg m$ is expected, the problem is underdeterminded. However, under suitable conditions on $A$, we can exploit the sparsity of $x$ to recover $(c_{\theta_j},\theta_j)_{j=1}^s$.

A necessary condition to do so is that if we have given a second sparse vector $y\in\C^N$, we have to have
\begin{align*}
	Ax \approx Ay~ \Rightarrow~ x\approx y.
\end{align*}
This is satisfied if the matrix $A$ satisfies for a $\delta \in [0,1)$ the lower bound
\begin{align}\label{eqRIPlow} (1-\delta)\|z\|_2^2 \leq \|Az\|_2^2 \qquad \textrm{ for all $2s$-sparse vectors }z,\end{align}
which is just the lower bound of the well-known restricted isometry property. It can be directly applied, as $x-y$ is $2s$-sparse. 

One simple consequence of \eqref{eqRIPlow} is that if all non-zero coefficients of $x$ and $y$ have a modulus of at least $\rho>0$, we have that
\begin{align} \label{eqSuppEqual} \|A(x-y)\|_2^2 < (1-\delta) \rho^2 ~\Rightarrow~ \supp x = \supp y.\end{align}
Combined with a noise model, one can easily derive a posteriori error estimates. Of course all these considerations do not help too much in actually solving the problem $Ax = b$ with $x$ sparse, which is due to its non-convex nature quite difficult. However, without a bound like \eqref{eqRIPlow}, the problem is ill-posed and its solution non-unique or not stable. See \cite{foucart2013mathematical} for a thorough introduction to compressive sensing.

In practice, discretizing $\Theta$ poses a difficult problem. When discretizing too coarsely, the assumption $\theta_j \in \tilde{\Theta}$ might be far from true. On the other hand if one discretises $\Theta$ very finely, many columns of $A$ will be almost equal and \eqref{eqRIPlow} will not hold. There are possibilities to circumvent these problems. Though not of interest for our approach, we give two examples. An adaptive discretization strategy is possible, see for example \cite{hu2012compressed}.
Another alternative is to discretize $\Theta$ finely, and then stick to subspaces where all $\tilde{\theta}_j$s are not too close. That approach is called structured compressed sensing, see for example \cite{duarte2013spectral}.

Here, however, we strive to prove similar stability properties without having to discretize $\Theta$ first. Due to the continuity of $h$ it is not possible to stably distinguish between $h(\theta)$ and $h(\theta+\varepsilon)$; indeed, if $h$ is differentiable, we expect that
$$\|h(\theta)- h(\theta+\varepsilon)\| \sim |\varepsilon|,\quad \text{ for } \varepsilon \text{ small}.$$  
The natural idea is therefore to replace a sparse sum by a well-separated sum. Then one can ask whether it is possible to estimate $\{ \theta_1,\dots, \theta_s \}$ stably, when there is a separation distance $q\in\R_{>0}$ such that  $|\theta_j-\theta_k|\geq q$ for all $j\neq k$. 

On the first glance, it seems that we have to prove a bound like $\eqref{eqRIPlow}$. Such a bound would be of the form
\begin{align}\label{eqWellSepRip1} \|c\|_2^2 \lesssim \left\| \sum_{\theta\in Y} c_\theta h(\theta)\right\|^2
\quad \end{align}
uniformly over all $Y\subset \Theta$ satisfying $|\theta-\theta'|\geq q$ for all $\theta\neq \theta'\in Y$. This, however, is not sufficient to guarantee that we are able to distinguish two different well-separated families. Let $Y'=\{ \theta_1',\dots, \theta_{s'}' \}$ be a second set with $q$ separated $\theta_j$s. We cannot hope to get a bound like \eqref{eqWellSepRip1}, as some $\theta_j'$ could be arbitrarily close to a $\theta_k$. 

The next best thing is that a good matching between $Y$ and $Y'$ must exist if their samples are very close. Clearly, for every $\theta \in Y$ there is at most one $\theta'\in Y'$ with $|\theta - \theta'| < q/2$, which we call $n(\theta)$. We denote by $Y_1$ all $y\in Y$ with such a match and by $Y_3$ all elements from $Y\cup Y'$ without a match (the enumeration is chosen to be consistent with later use). Then a bound like
\[
\sum_{\theta\in Y_3} |c_\theta|^2 + \sum_{\theta\in Y_1}\left( |\theta-n(\theta)|^2|c_\theta + c_{n(\theta)}|^2 + |c_\theta - c_{n(\theta)}|^2 \right) \lesssim \left\| \sum_{\theta \in Y} c_\theta h(\theta) - \sum_{\theta' \in Y'} c_{\theta'}h(\theta')\right\|^2 \]
might be possible. If the right-hand side is sufficiently small and all coefficients have a modulus of at least $c_{\min} > 0$, we can deduce that $Y_3 = \emptyset$ and a one-to-one matching between $Y$ and $Y'$ exists. That can be interpreted as the continuous analog of \eqref{eqSuppEqual}. Furthermore, if the right-hand side becomes very small, we can conclude that $\theta\approx n(\theta)$ and $c_\theta \approx c_{n(\theta)}$. 

To prove such a bound for the frequency estimation problem is the main objective of this paper. The bound is presented in Theorem \ref{mainThm1d} and its implication for the well-posedness is given in Corollary \ref{corWellPosedness}. An example of an a posteriori bound is given in Corollary \ref{corAPost}. Clearly, such a bound always implies a lower bound of the corresponding sampling matrix, which is a Vandermonde matrix in our case. We state these bounds explicitly in Section \ref{secSingValVander}.

Such  bounds do not give a tractable method to determine $Y$. They only indicate that a stable estimation is possible. However, for the problem at hand, many very efficient methods exist. Some of them, like ESPRIT \cite{roy1986esprit}, do not need a discretization of $Y$, as they estimate $Y$ by an eigenvalue problem. 

The results presented here are not directly related to the popular approach of using TV-minimization to solve the sparse frequency estimation problem. We are interested in proving intrinsic stability of the problem, independent of the solution method. On the other hand, TV-minimization is one particular technique. Unfortunately, it is not possible to directly reprove results from TV-minimization using the bounds presented here. The reason is that one cannot easily prove that TV-minimization yields measures satisfying our model assumptions. We refer to \cite{duval2015exact} for a good overview, including stability and convergence results. 

\section{Localizing Functions and Stability} \label{secLocFuncAndStab}
We introduce some notation. Let
\begin{align*}
	\mcs = \left\{ \sum_{y\in Y} c_y e^{2\pi i y \cdot} ~:~ c_y\in\C\setminus \{0\} ,~Y \subset [0,1) \textrm{ finite}\right\}.
\end{align*}
For an $f\in \mcs$, we call $Y^f$ the set of its frequencies and $c^f\in\C^{Y^f}$ the corresponding coefficients, which we enumerate by $Y^f$. 
We measure the distance of two frequencies $y,y'\in [0,1)$ by the \textit{wrap-around distance}
$$ |y-y'|_\T = \min_{k\in\Z} |y-y'-k|$$
and for a finite set $Y\subset [0,1)$, we define its separation by
$$ \sep Y = \min_{\substack{y,y'\in Y\\ y\neq y'}} |y-y'|_\T$$
and collect for a $q>0$ all well-separated exponential sums in 
$$ \mcs(q) = \left\{f\in\mcs ~:~ \sep Y^f \geq q \right\}.$$
As a first result, we prove bounds of the form 
$$\| c^f\|_2^2 \lesssim \sum_{k=A}^B |f(k)|^2 \lesssim \|c^f\|_2^2 \qquad \forall f\in\mcs(q).$$
Such results are well-known, particular sharp lower bounds are proven by Moitra in \cite{moitra2015super} and by Aubel and B\"olcskei in \cite{aubel2017vandermonde}. A sharp upper bound was already given by Selberg \cite{vaaler1985some}. We use basically the same approach as Moitra, however improving it. 

At the core of Moitra's proof are particular functions, that were already constructed by Selberg.  These functions $\psi_{A,B,q}\in L^1(\R),~A,B\in\Z,~A<B,~q\in\R_{>0}$ satisfy the following properties:
\begin{enumerate}[label=(P\arabic*)]
	\item \label{prop_loc1} $\psi_{A,B,q}\leq \chi_{[A,B]}$,
	\item \label{prop_loc2} $\supp \mcf \psi_{A,B,q} \subset [-q,q]$,
	\item $\mcf \psi_{A,B,q} (0) = \int_\R \psi_{A,B,q}(x)\td x = B-A - q^{-1}.$
\end{enumerate}
Note that \ref{prop_loc2} implies that $\psi_{A,B,q}$ is an entire function. 

These functions $\psi_{A,B,q}$ are useful, because they allow to estimate a function localized in the spatial domain by something localized in the frequency domain, thus ``cheating'' the uncertainty principle. We call any function satisfying \ref{prop_loc1} and \ref{prop_loc2} \textit{$q$-localizing functions}. For the special case $B-A\in q\Z$, the functions $\psi_{A,B,q}$ are extremal, in the sense that $\int_\R \chi_{[A,B]}- \psi_{A,B,q}$ is minimized. The extremal functions in the other cases are known as well, a result due to Littmann, see \cite[Theorem 5.2]{littmann2013quadrature}.

The localizing function lie in the Paley-Wiener space
\[ \PW = \{ f \in L^2(\R) ~:~ \supp \mcf f \subset [-1,1] \}.
\]
Furthermore, the following version of Poisson summation will prove useful, which follows from the more general Poisson formula given in \cite[p. 69]{zygmund1968trigonometric}.

\begin{theorem}[Poisson Summation Formula]
	For any $f\in \PW\cap\, L^1(\R)$ it holds true that
	$$\sum_{k\in\Z} f(k) = \sum_{k\in\Z}\hat{f}(k).$$
\end{theorem}

With these ingredients, we can give an improved version of Moitra's lower bound.
 
\begin{theorem} \label{tWellSepExp}
	Let $f\in\mcs(q)$ and $A,B\in\Z$, $A<B$. Then 
	\begin{align*}
		\left(B-A+2-\frac1{q}\right)\|c^f\|_2^2 \leq \sum_{k=A}^B |f(k)|^2 \leq \left( B-A +\frac1{q} \right).
	\end{align*}
\end{theorem}
\begin{proof}
	The upper bound is already due to Selberg, see \cite{vaaler1985some}. For the lower bound, note that
	\begin{align*}
		\sum_{k=A}^B |f(k)|^2 &\geq \sum_{k\in\Z} \psi_{A-1,B+1,q}(k)|f(k)|^2 = \sum_{y,y'\in Y^f} c_y \overline{c}_{y'} \sum_{k\in\Z} \psi_{A-1,B+1,q}(k)e^{2\pi i (y-y')k} \\
		&= \sum_{y,y'\in Y^f} c_y \overline{c}_{y'} \hat{\psi}_{A-1, B+1, q}(y-y') = \left(B-A+2-\frac1{q}\right)\|c^f\|_2^2,
	\end{align*}
	where we used Poisson summation and the fact that $\psi_{A-1, B+1, q}(A-1)=\psi_{A-1, B+1, q}(B+1) \leq 0$ (in fact, equality holds) due to \ref{prop_loc1} and the continuity of $\psi_{A-1,B+1,q}$.
\end{proof}

\begin{remark}
	Moitra used $\psi_{A,B,q}$ instead of $\psi_{A-1,B+1,q}$, resulting in the constant $(B-A-q^{-1})$. Aubel and B\"olcskei improved the constant to $(B-A+\frac32 - q^{-1})$. Furthermore, they discuss the more general case of frequencies in the unit disc. The constant $\left(B-A+2-\frac1{q}\right)$ is sharp in the following sense: For $q= (B-A+2)^{-1}$ the space $\mcs(q)$ contains a linear subspace of dimension $B-A+2$ and it is therefore clear that there is a $f\in \mcs(q)$ vanishing on $k=A,\dots, B$. Thus, for $q= (B-A+2)^{-1}$ the best lower bound is zero. 
\end{remark}

The proof of Theorem \ref{tWellSepExp} is worth a short reflection. A $q$-localizing function $\psi$ gives rise to a sesquilinear form 
\[ (\cdot, \cdot)_\psi : \mcs \times \mcs \rightarrow \C
\]
satisfying the following properties:
\begin{enumerate}
	\item For two exponentials, we have
	\[ (e^{2\pi i y \cdot}, e^{2\pi i y'})_\psi =\hat{\psi}(y-y'). \]
	In particular, it is local in the sense that whenever $|y-y'|_\T \geq q$, we have that
	\[ (e^{2\pi i y \cdot}, e^{2\pi i y'\cdot})_\psi = 0.  \]
	\item It minorizes sampling at $A,\dots, B$, i.e., for all $f\in\mcs$ we get 
	 \[ (f, f)_\psi \leq \sum_{k=A}^B |f(k)|^2. \]
\end{enumerate}
One can use these properties to estimate other constellations of frequencies. If for an $f\in\mcs$ we have that
\[ f = f_1+\dots+ f_R \]
such that all $f_r \in \mcs$ have mutually $q$-separated frequencies in the sense that
\[\dist(Y^{f_j}, Y^{f_k}) \geq q, \qquad \textrm{ for all } j\neq k,\]
we obtain the lower bound
\[  \sum_{k=A}^B |f(k)|^2 \geq \sum_{r=1}^R (c^{f_r})^* \left( \hat{\psi}(y-y') \right)_{y,y'\in Y^{f_r}} c^{f_r} \geq \lambda_{\min} \|c^f\|_2^2. \]
Here, $\lambda_{\min}$ denotes the smallest eigenvalue of any of the $R$ matrices $( \hat{\psi}(y-y') )_{y,y'\in Y^{f_r}}$.
For example if every $f_j$ corresponds to a cluster of up to $L$ frequencies, we have to find a $\psi$ such that 
\begin{align} \left(\hat{\psi}(y-y') \right)_{y,y'\in Y} \label{eqKernelMatPsi}\end{align}
is positive definite for any $Y\subset [0,1)$, $|Y|\leq L$ of interest, e.g., of a certain inner separation. Then we need to bound the smallest eigenvalue uniformly over all these sets.

For general values of $L$ that seems to be a difficult problem. However, in the case $L=2$, this is easy, as for $Y=\{y,y'\}$ the matrix \eqref{eqKernelMatPsi} has eigenvalues $\hat{\psi}(0)\pm |\hat{\psi}(y-y')|$. Clearly, it is necessary for $\hat{\psi}$ to have a global maximum in zero. Then, an estimate of the form
\begin{align}\label{eqLocPsiHat} \hat \psi (0) - |\hat{\psi}(y-y')| \gtrsim |y-y'|_\T^2
\end{align}
gives rise to a very sharp bound. The localizing functions used above unfortunately do not have a global maximum at zero. We construct an alternative. A useful tool is a formula going back to Jagerman and Fogel \cite{jagerman1956some}, which allows for Hermite interpolation on $\Z$ in the Paley-Wiener space $\PW$.

\begin{theorem}
	For any $f\in \PW$ the following representation holds true:
	\[ f(x) = \frac{\sin^2(\pi x)}{\pi ^2} \sum_{k\in\Z} \left( \frac{f(k)}{(x-k)^2}+\frac{f'(k)}{x-k}\right).
	\]
\end{theorem}

We are now able to construct a suitable localizing function. 

\begin{lemma}
	The function $\phi\in\PW \cap \, L^1(\R)$ defined by
	\[  \phi(x) = \frac{\sin^2(\pi x)}{\pi ^2} \left( \frac23 \left( \frac1{x}-\frac1{x-3} \right)+ \frac1{(x-1)^2}+\frac1{(x-2)^2} \right) \]
	satisfies $\phi\leq \chi_{[0,3]}$ and $\hat{\phi}(0) = 2$.
\end{lemma}
\begin{proof}
	We start with $\phi\leq \chi_{[0,3]}$. The claim obviously holds true for $x\in\Z$, which we exclude in the following. Verifying $\phi(x)\leq 0$ for $x\leq 0$ is a direct calculation that we omit here. The case $x\geq 3$ follows by symmetry.
	
	The inequality $\phi(x)\leq 1$ for $x\in [0,3]$ is slightly more complicated. Note that
	\begin{align*}
		\phi(x) \leq 1 = \frac{\sin^2(\pi x)}{\pi^2} \sum_{k\in\Z} \frac1{(x-k)^2},
	\end{align*}
	which is equivalent to 
	\begin{align}\label{eqIneqProofLemma}
		\frac23 \left( \frac1{x}+\frac1{3-x}\right) \leq \sum_{\substack{k\in\Z \\ k\neq 1,2}} \frac1{(x-k)^2}.
	\end{align}
	Now as the trapezoidal rule overestimates convex functions, we can estimate for $K\in\Z$ and $x < K$ that
	\begin{align*} \sum_{k\geq K} \frac1{(x-k)^2} \geq \frac12\frac1{(x-K)^2} + \int_K^\infty \frac1{(x-y)^2}\td y = \frac12\frac1{(x-K)^2} +\frac1{K-x}
	\end{align*}
	and \eqref{eqIneqProofLemma} follows by applying this twice. $\hat{\phi}(0) = 2$ results from Poisson summation.
\end{proof}

\begin{remark}
	A few remarks are in order.
	\begin{enumerate}
		\item It follows by Poisson summation that $\phi$ maximizes $\hat{\phi}(0)$ over all $f\in\PW$ satisfying $f\leq \chi_{[0,3]}$. $\phi$ is not unique, other values for $\phi'(0)$ than $\frac23$ are possible (namely, any number between $\left[\frac23, \frac43 \right]$). Similar constructions work for any minorants of $\chi_{[A,B]},~ A,B\in\Z$. In particular, it follows that if $A+1=B$, the constant zero minorant is optimal.
		\item The non-uniqueness is certainly already known, however difficult to track down. For example, in \cite[p. 289]{graham1981class} a similar result is cited for the majorants and attributed to unpublished work of Selberg. 
		\item The construction originally used by Selberg and by Moitra corresponds to choosing $\phi'(0) = 1$. 
	\end{enumerate}
\end{remark}

\eqref{eqLocPsiHat} is now the result of a direct calculation. 
\begin{lemma}
	$\phi$ satisfies
	\[ \hat{\phi}(0) -|\hat{\phi}(w)| \geq \begin{cases} \pi^2 w^2, & \textrm{ for } |w|\in \left[0,\frac13\right]\\
				\pi^2/9 & \textrm{ for } |w|\in \left[\frac13,1\right]. \end{cases}
			 \]
\end{lemma}
\begin{proof}
	The Fourier transform of $\phi$ is given by
	\[\hat{\phi}(w) = (1-w)\left( e^{-2\pi i w } + e^{-4\pi iw } \right) +\frac{1}{3\pi i} \left(1-e^{-6\pi iw} \right),\quad \textrm{ for }w\in [0,1] \]
	and $\hat{\phi}(w) = \overline{\hat{\phi}(-w)}$ for $w\in[-1,0]$.
	
	We estimate $|\hat{\phi}|$ by first noting that for $w\in \left[0, \frac13\right]$ we have
	\begin{align*}
	|\hat{\phi}(w)| &= \left|e^{3\pi i w}\hat{\phi}(w)\right| = \left| (1-w) \left(e^{\pi i w} + e^{-\pi i w}\right) + \frac{1}{3\pi i} \left(e^{ 3\pi i w}-e^{-3\pi i w}\right)\right| \\
	&=(1-w)2\cos(\pi w) + \frac{2}{3\pi}\sin(3\pi w) .
	\end{align*}
	Next, we use 
	$$\cos(x) \geq 1-\frac{x^2}{2} \quad \text{ and } \quad \sin(x) \geq x - \frac{x^3}{6} \quad \forall x\in \R_{\geq 0}$$
	and obtain
	$$|\hat{\phi}(w)|\geq (1-w)\left(2- \pi^2w^2\right)+2w-3\pi^2 w^3 = 2-\pi^2w^2-2\pi^2w^3.$$
	This results in 
	\begin{align*}  \hat{\phi}(0)-|\hat{\phi}(w)| \geq \pi^2 w^2 \quad \text{ for } w\in \left[0, \frac13\right].
	\end{align*}
	It is not difficult to show that $\hat{\phi}(0)-|\hat{\phi}(w)| \geq \pi^2/9$ for $w\in \left[\frac13,1\right]$. We omit the computation here.
\end{proof}

%fourier transform of -3*sin(\pi * x)^2/(\pi ^2x (x-3))
%Plot[ |(1-w)\left( e^{2\pi i w } + e^{4\pi iw } \right) +\frac{i}{3\pi} \left(1-e^{6\pi iw} \right)|  , {w,0,1}]
Now we are able to state the main theorem.

\begin{theorem} \label{mainThm1d}
	Let $N\in \N$ be given. Further, let $f\in \mcs$ be such that $Y^f$ can be decomposed in three disjoint sets $Y_1, Y_2, Y_3$ satisfying:
	\begin{enumerate}
		\item For all $y,y'\in Y_j$ with $y\neq y'$ we have that $|y-y'|_\T \geq  \frac3{N+1}$. 
		\item For all $y\in Y_1$ there is exactly one $y'\in Y_2$ with $|y-y'|_\T < \frac3{N+1}$. We denote $y'$ by $n(y)$.
		\item For all $y\in Y_3$ and all $y'\in Y_1\cup Y_2$ we have that $|y-y'|_\T \geq  \frac3{N+1}$. 
	\end{enumerate}
	Then the following bound holds true:
	\begin{align*}
	\sum_{k=1}^{N} |f(k)|^2 \geq \frac23 &(N+1)\sum_{y\in Y_3} |c_y|^2 + \sum_{y\in Y_1}\left[ 
	\frac{N+1}3\left| e^{-3\pi i y}c_y+  e^{-3\pi i n(y)} c_{n(y)} \right|^2 \right.\\
	&\left.+ \frac{\pi^2(N+1)^3}{2\cdot 3^5} |y-n(y)|^2 \left| e^{-3\pi i y}c_y -  e^{-3\pi in(y)} c_{n(y)} \right|^2 \right].
	\end{align*}
\end{theorem}
\begin{proof}
	We dilate $\phi$ to fit our data. Let 
	$$ \phi_N(x) = \phi\left(\frac{3x}{N+1}\right) \leq  \chi_{[0,N+1]}(x),\qquad 
	\hat{\phi}_N(w) = \frac{N+1}{3}\hat{\phi} \left( \frac{(N+1)w}{3} \right). $$
	Note that $\supp \hat{\phi}_N \subset \left[ -\frac3{N+1}, \frac3{N+1}\right]$. Again, we use $\phi_N$ to localize the sum:
	\begin{align} \nonumber
		\sum_{k=1}^{N} |f(k)|^2 &\geq \sum_{k\in \Z} \phi_N(k)|f(k)|^2 = \sum_{y,y'\in Y^f} c_y\overline{c}_{y'}\hat{\phi}_N(y-y')\\
		\label{eqProofMainThmII}
		&= \hat{\phi}_N(0)\sum_{y\in Y_3} |c_y|^2 + \sum_{y\in Y_1} \bldc_y^* \bldA_y
		\bldc_y
	\end{align}
	with 
	$$\bldc_y=\begin{pmatrix}c_y\\c_{n(y)}\end{pmatrix}\quad\text{ and } \quad \bldA_y=\begin{pmatrix}
	\hat{\phi}_N(0) & \hat{\phi}_N(n(y)-y) \\ \overline{\hat{\phi}_N(n(y)-y)} & \hat{\phi}_N(0)\end{pmatrix}.$$
	Next, we consider the eigenvalue decomposition of $\bldA_y$. We have the two eigenvectors
	$$ \bldv_1 = \frac1{\sqrt{2}} \begin{pmatrix}
	1 \\ e^{-i \Arg \hat{\phi}_N(n(y)-y) )} 
	\end{pmatrix}
	\text{ and }
	\bldv_2 = \frac1{\sqrt{2}} \begin{pmatrix}
	1 \\ -e^{-i \Arg \hat{\phi}_N(n(y)-y) )} 
	\end{pmatrix}.
	$$
	Letting $\theta_y = \Arg \hat{\phi}_N(n(y)-y) ) = 3\pi i (n(y)-y)$, we get
	\begin{align*}2\bldc_y^*\bldA_y \bldc_y = &\left( \hat{\phi}_N(0) + |\hat{\phi}_N(n(y)-y)|  \right) \left| c_y+  e^{-i \theta_y} c_{n(y)} \right|^2 \\&+\left( \hat{\phi}_N(0) - |\hat{\phi}_N(n(y)-y)|  \right) \left| c_y -  e^{-i \theta_y} c_{n(y)} \right|^2 
	\end{align*} 
	Plugging that into \eqref{eqProofMainThmII} and using Lemma \ref{mainThm1d} we obtain
	\begin{align*}
		\sum_{k=1}^{N} |f(k)|^2 \geq \frac23 &(N+1)\sum_{y\in Y_3} |c_y|^2 + \sum_{y\in Y_1}\left[ 
		\frac{N+1}3 \left| c_y+  e^{-i \theta_y} c_{n(y)} \right|^2 \right.\\
		&\left.+ \frac{\pi^2(N+1)}{54} \min \left\{
	(N+1)^2|y-n(y)|^2, 1 \right\} \left| c_y -  e^{-i \theta_y} c_{n(y)} \right|^2 \right].
	\end{align*}
	Using $|y-n(y)| < \frac3{N+1}$ gives the result.
\end{proof}

As the bound presented in Theorem \ref{mainThm1d} is rather technical, we discuss it and give a few special cases and variations. Note that we did not try to optimize the constants. Slightly sharper results are certainly possible. 

\textbf{Sharpness}: The estimate is sharp for $|y-n(y)|\rightarrow 0$ up to constants. To see that, we consider
$$f_\tau(x) = 1- e^{2\pi i \tau x}.$$
A direct estimate gives for $|\tau| \ll N^{-1}$
$$ \sum_{k=1}^N |f(k)|^2 \sim \sum_{k=1}^N \left|2\pi i \tau k \right|^2 \sim \tau^2 N^3,$$
which is the exact asymptotic Theorem \ref{mainThm1d} gives as well.
Clearly, for $\|c^f\|_2 \rightarrow 0$ the estimate is sharp as well.

The estimate gives the expected result for the signs of the coefficients as well. Namely, if \mbox{$c_y\approx -c_{n(y)}$} the critical cancellation occurs, as the first term of the second sum vanishes. If one is interested in an estimate independent of $\Arg c$, one obtains
\begin{align} \label{eqMainThmWeakend}
\begin{aligned}
\sum_{k=1}^{N} |f(k)|^2 \geq \frac23 &(N+1)\sum_{y\in Y_3} |c_y|^2 
+ \frac{\pi^2(N+1)^3}{2\cdot 3^5} \sum_{y\in Y_1}|y-n(y)|^2\left( |c_y|^2+|c_{n(y)}|^2 \right).
\end{aligned}
\end{align}

\textbf{Symmetric Samples}: Quite often, in particular when $N$ is odd, one is interested in estimating
$$\sum_{k=1}^N \left| f\left(k - \frac{N+1}{2}\right) \right|^2.$$
Due to symmetry, one obtains the estimate
\begin{align}\label{eqSymSamp}
\begin{aligned}
\sum_{k=1}^N \left| f\left(k - \frac{N+1}{2}\right) \right|^2 \geq \frac23 &(N+1)\sum_{y\in Y_3} |c_y|^2 + \sum_{y\in Y_1}\left[ 
\frac{N+1}3 \left| c_y+  c_{n(y)} \right|^2 \right.\\
&\left.+ \frac{\pi^2(N+1)^3}{2\cdot 3^5} |y-n(y)|^2 \left| c_y -  c_{n(y)} \right|^2 \right].
\end{aligned}
\end{align}
Here, we do not need to modulate the coefficients $c_{n(y)}$ and $c_y$.

\textbf{Choice of localizing function}: Different choices for $\phi$ are possible. Assume we pick a 1-localizing function $\phi \leq \chi_{[0,A]}$. Then, after dilation, we can localize up to $A/(N+1)$. As discussed earlier, $A> 1$ is necessary. Using the extremal construction outlined above for $A=2$ instead of $A=3$, gives worse rates, as it can be shown that it only satisfies 
$$\hat{\phi}(0)-|\hat{\phi}(w)| \gtrsim |w|^3.$$
Note that we cannot hope to obtain a bound as in Theorem \ref{mainThm1d}, if $q < \frac2{N+1}$, as then counterexamples exist: $(N+1)/2$ pairs of very close frequencies would result in $|Y^f|=N+1$ and such an exponential sum can vanish on $N$ points. 

In \cite{diederichs2018sparse} the function $\phi \leq \chi_{[-2,2]}$ was used, covering the symmetric case. Due to that, the assumptions on $q$ are a slightly stronger there.

\section{Singular Value Estimates of Vandermonde Matrices}\label{secSingValVander}

Both Theorem \ref{tWellSepExp} and Theorem \ref{mainThm1d} directly give estimates of the smallest singular values of Vandermonde matrices. For the reader's convenience, we state them here explicitly.

We denote the Vandermonde matrix with frequencies $y_1,\dots, y_M \in [0,1)$ by
$$
V_N(y_1,\dots, y_M) := \begin{pmatrix}
1 & 1 & \dots & 1\\
e^{2\pi i y_1} & e^{2\pi i y_2} &\dots & e^{2\pi i y_M} \\
\vdots & \vdots & \ddots & \vdots \\
e^{2\pi i (N-1)y_1} & e^{2\pi i (N-1)y_2}  &\dots & e^{2\pi i (N-1)y_M} 
\end{pmatrix} \in \C^{N\times M}.
$$
Using
\[ 
V_N(y_1,\dots, y_M)c = \left(\sum_{j=1}^{M}c_j e^{2\pi i y_j k}\right)_{k=0,\dots, N-1},
\]
Theorem \ref{tWellSepExp} gives the following:
\begin{corollary}
	Let $y_1\dots, y_M\in [0,1)$ be $q$-separated frequencies. Then the smallest singular value $\sigma_{\min}$ of $V_N(y_1,\dots, y_M)$ satisfies
	\[ \sigma_{\min}^2 \geq \left( N +1 - \frac1{q} \right). \]
\end{corollary}
Analogously,\eqref{eqMainThmWeakend} gives rise to the following result for Vandermonde matrices with pairwise colliding nodes.

\begin{corollary} \label{cVanderPairs}
	Let $Y = \{y_1\dots, y_M \} \subset [0,1)$ and $Y = \{y_1',\dots, y_{M'}' \} \subset [0,1)$ be two $q$-separated sets of frequencies, where $q\geq \frac3{N+1}$. Assume further that for each $y\in Y$ there is at most one $y'\in Y'$ such that $|y-y'|_\T < q$. Let $\tau$ be the smallest of these distances, i.e., $\tau = \dist_\T(Y, Y')$. Then the smallest singular value $\sigma_{\min}$ of $V_N(y_1,\dots, y_M, y'_1,\dots, y_{M'}')$ satisfies
	$$\sigma_{\min}^2 \geq \frac{\pi^2}{2\cdot 3^5}(N+1)^3 \tau^2,$$
	assuming that $\tau < \frac3{N+1}$ and 
	$$\sigma_{\min}^2 \geq \frac23(N+1)$$
	otherwise.
\end{corollary}

The corollary is a significant improvement over similar estimates, very recently presented in \cite{kunis2018condition}. Bounding singular values of Vandermonde matrices with multiple tightly clustered sets of frequencies has attracted some attention recently, as it gives insights when recovery of frequencies even without separation might be possible. Corollary \ref{cVanderPairs} is (up to small improvements in the constants) sharp for the case of clusters of two nodes. For results covering clusters of multiple nodes, see \cite{batenkov2018stability, li2017stable}.

Finally, note that Theorem \ref{mainThm1d} gives some information on the geometry of the singular spaces associated with small singular values. Indeed, it confirms the intuition that close frequencies should have coefficients that sum to zero.

\section{Well-posedness and A Posteriori Error Estimates} \label{secAPostEst}

Next, we apply Theorem \ref{mainThm1d} to obtain conditional well-posedness of the frequency estimation problem. To simplify notation, we stick to the case of symmetric samples $f(k),~k= -N, \dots, N$. 

 Assume that we are given $f,g\in\mcs(2q)$. Then for every $y\in Y^f$ there is at most one $y'\in Y^g$ with $|y-y'|_\T < q$. Therefore, we are in position to apply Theorem \ref{mainThm1d}. 
 
 \begin{corollary} \label{corWellPosedness}
 	Let $f,g\in \mcs(2q)$ and $N\in \N_{>0}$ with $q\geq \frac3{2N+2}$ be given. Assume further that all coefficients of $f$ and $g$ have a modulus of at least $c_{\min}\in \R_{>0}$. If
 	$$\sum_{k=-N}^N |f(k)-g(k)|^2 < \frac{4N+4}{3}c_{\min}^2, $$
 	then for every $y\in Y^f$ there is exactly one $n(y)\in Y^g$ with $|y-n(y)|_\T < \frac3{2N+2}$. 
 	
 	Furthermore, the following estimate holds true:  
\begin{align*}
\sum_{y\in Y^f}\left[ 
\frac{N+1}{3}\left| c_y -  c_{n(y)} \right|^2 + \frac{2\pi^2(N+1)^3}{3^5}|y-n(y)|^2 \left| c_y +  c_{n(y)} \right|^2 \right] \leq \sum_{k=-N}^N \left| f(k)-g(k)\right|^2.
\end{align*}
\end{corollary}

\begin{proof}
	We invoke \eqref{eqSymSamp} to see that $Y_3 = \emptyset$, i.e., for all $y\in Y^f$ there has to be a $n(y)$. The claim follows by using that $|y-n(y)|_\T < \frac3{2N+2}$.
\end{proof}

That estimate can now easily be used to obtain a posteriori error estimates. Assume that we are given noisy samples $\tilde{f}(k)= f(k)+ \eta_k$ and use any algorithm we want to obtain a candidate $g$. Now, if both fit our model, i.e., $f,g\in \mcs(2q)$, Corollary \ref{corWellPosedness} applies. However, as we do not know $f(k)$, we cannot use  Corollary \ref{corWellPosedness} directly. Instead, we combining it with a noise model to get an a posteriori estimate. 

As the noise model depends on the application, we give a prototypical result, using Gaussian noise. 

\begin{lemma}
	Let $\bldv \in \C^{K}$ and $\bldeta = (\eta_k)_k\in\C^K$, with
	$$\eta_k = X_{k,1}+ i X_{k,2},$$
	where $X_{k,j}$ are independent Gaussian random variables with mean zero and variance $\sigma^2$. Let $\delta \in (0,1)$. Then with probability 
	$$ 1-e^{-K^{(1+\delta)/2}}-2e^{-K^\delta /8}$$
	it holds true that
	$$\|\bldv \|_2 \leq \left|\|\bldv + \bldeta\|_2^2 -2K\sigma^2\right|^{1/2} + (2+\sqrt{2})\sigma K^{(1+\delta)/4}.$$
\end{lemma}
\begin{proof}
	\[ \label{eqProofLemmaGauss}\|\bldv + \bldeta\|_2^2 = \|\bldv\|_2^2 + \sigma^2 Y + 2\sigma \|\bldv \|_2 Z,\]
	where $Z$ is a standard Gaussian random variable and $Y$ is a $\chi^2$-distributed with $2K$ degrees of freedom. \eqref{eqProofLemmaGauss} is equivalent to 
	\[ \|\bldv + \bldeta\|_2^2+\sigma^2 Z - \sigma^2Y = (\|v\|_2^2 + \sigma Z)^2.\]
	We estimate
	\begin{align*} \|\bldv\|_2 &\leq \left( \sigma^2Z^2 -\sigma^2Y+\|\bldv+\bldeta\|_2^2\right)^{1/2} +\sigma |Z|
	\\ &\leq \left| \|\bldv +\bldeta\|_2^2-2K\sigma^2\right|^{1/2} +2\sigma |Z| +\sigma |Y-2K|^{1/2},
	\end{align*}
	where we used the subadditivity of the square root.
	
	Next, we use two tail estimates. For $Z$ we have that
	$$ \Pr (|Z|\geq t_1 ) \leq e^{-\frac{t_1^2}{2}}\quad \text{ for all } t> 0, $$
	see for example \cite{foucart2013mathematical}, Proposition 7.5. And $Y$ is concentrated around $\E Y =2K$:
	$$\Pr( |Y-2K|\geq 2Kt_2) \leq 2 e^{-Kt_2^2/4} \quad \text{ for all } t_2\in (0,1), $$
	see \cite{bercu2015concentration}, Theorem 2.57. Now we choose
	$$ t_1 = K^{(1+\delta)/4}, \quad t_2 = \frac{K^{(\delta-1)/2}}{\sqrt{2}}$$
	and by the union bound we obtain with probability
	$$1-e^{-K^{(1+\delta)/2}}-2e^{-K^\delta /8}$$
	that 
	$$\|\bldv \|_2 \leq\left|\|\bldv + \bldeta\|_2^2 -2K\sigma^2\right|^{1/2} + (2+\sqrt{2})\sigma K^{(1+\delta)/4}.$$
\end{proof}

Combining these results, we obtain the following error estimate.

\begin{corollary} \label{corAPost}
	Assume that $f,g\in \mcs(2q)$ with $q \geq \frac3{2N+2}$. Let
	$$\bldf = (f(k))_{k=-N,\dots, N} \in \C^{2N+1},~ \bldg = (g(k))_{k=-N,\dots, N}\in \C^{2N+1},~ \bldeta = (\eta(k))_{k=-N,\dots, N}\in \C^{2N+1}.$$
	Here,	
	$$\eta_k = X_{k,1}+ i X_{k,2},$$
	where $X_{k,j}$ are independent Gaussian random variables with mean zero and variance $\sigma^2$. Let $\delta \in (0,1)$. Then with probability 
	$$ 1-e^{-(2N+1)^{(1+\delta)/2}}-2e^{-(2N+1)^\delta /8}$$
	it holds true that if
	$$\left( \left| \|\bldf + \bldeta-\bldg \|_2^2 -2(2N+1)\sigma^2\right|^{1/2} + (2+\sqrt{2})\sigma (2N+1)^{(1+\delta)/4}\right)^2\leq \frac{4N+4}{3}c_{\min}^2$$
	we have that for every $y\in Y^f$ there is exactly one $n(y)\in Y^g$ with $|y-n(y)|_\T < \frac3{2N+2}$. 
	
	Furthermore, the following estimate holds true:
	\begin{align}\label{eqAPostEst}
	\begin{aligned}
	\sum_{y\in Y^f}&\left[ 
	\frac{N+1}3 \left| c_y-  c_{n(y)} \right|^2 + \frac{2\pi^2(N+1)^3}{3^5}|y-n(y)|^2 \left| c_y +  c_{n(y)} \right|^2 \right] \\ 
	&\leq \left( \left|\|\bldf + \bldeta-\bldg \|_2^2 -2(2N+1)\sigma^2\right|^{1/2} + (2+\sqrt{2})\sigma (2N+1)^{(1+\delta)/4}\right)^2.
	\end{aligned}
	\end{align}
\end{corollary}

\begin{remark}
	The case of bounded instead of Gaussian noise is easier. We skip the details.
\end{remark}

We close this paper by giving a small example, how the bound can be applied. 

\begin{example}
	We consider the exponential sum $f$ with frequencies 
	$$Y^f = \{ 0.1, 0.3, 0.6, 0.9 \}, \quad c = [1.1, -1.1, 2, 2],$$
	where we enumerate the frequencies in order of their size. Now we sample $f$ at $-20, \dots, 20$ and add Gaussian noise as described in Corollary \ref{corAPost} of variance $\sigma^2$. Next, we apply ESPRIT to obtain an estimate $\tilde{f}$. We calculate the left-hand side of \eqref{eqAPostEst}, called error, and the right-hand side, called error estimator. Further, we calculate the sampling distance, which is $\|\bldf-\tilde{\bldf}\|_2^2$. The results for different choices of $\sigma$ are presented in Figure \ref{fErrEst01}. For each choice of $\sigma$, fifty instances were calculated and the largest error is shown. 
	
	We pick $\delta = 0.9$, resulting in a probability of a little bit more than $94\%$, that the estimate can be applied. For $\sigma=1$ the premise of Corollary \ref{corAPost} is actually not satisfied. We can observe in Figure \ref{fErrEst01} that the given bound seems to be reasonably sharp. 
	
	\begin{figure}
		\centering
	%	\hspace*{-3cm}
		\includegraphics[scale=0.5]{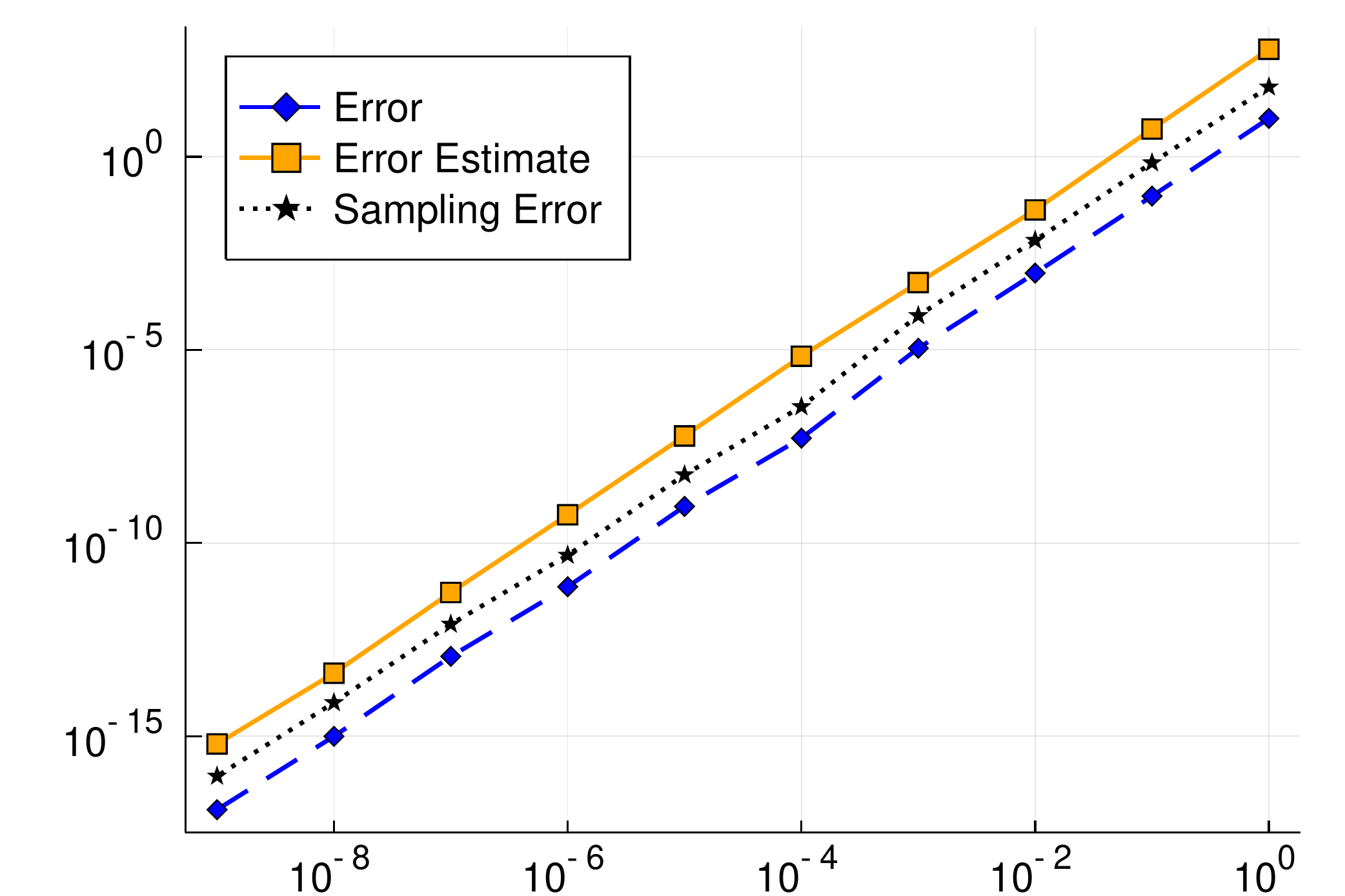}
	%	\vspace*{-1.2cm}
		\caption{Comparison of the error and the error estimate at different noise levels. $x$-axis: Standard deviation of the noise. }
		\label{fErrEst01}

	\end{figure}
	
\end{example}

\subsection*{Acknowledgments}
The author was supported by the Bavarian Ministry
of Economy by means of the research project ``Big Picture''. 

The results of the paper are extensions of results obtained in the author's phd thesis \cite{diederichs2018sparse}, written under supervision of Prof. Armin Iske. The author would like to thank Prof. Iske for his constant support and encouragements.
Furthermore, the author is indebted to Prof. Tomas Sauer for helpful remarks on the manuscript.

\FloatBarrier
\bibliographystyle{plain}
\bibliography{bibliography}

%
%\input{paper.bbl}
%\bibliographystyle{acm}  %{dinat} %{math}
%\addcontentsline{toc}{chapter}{Bibliography}
%\printbibliography

\end{document}